\documentclass[11pt]{amsart}
\usepackage[margin=1.35in]{geometry}
\usepackage{amscd,amsmath,amsxtra,amsthm,amssymb,stmaryrd,xr,mathrsfs,mathtools,enumerate,commath, comment}
\usepackage{stmaryrd}
\usepackage{xcolor}
\usepackage{commath}
\usepackage{comment}
\usepackage{orcidlink}
\usepackage{tikz-cd}
\usepackage{longtable} % for 'longtable' environment
\usepackage{pdflscape} % for 'landscape' environment
\usepackage{booktabs}
\usepackage{hyperref}
\definecolor{vegasgold}{rgb}{0.77, 0.7, 0.35}
\definecolor{darkgoldenrod}{rgb}{0.72, 0.53, 0.04}
\definecolor{gold(metallic)}{rgb}{0.83, 0.69, 0.22}
\hypersetup{
 colorlinks=true,
 linkcolor=darkgoldenrod,
 filecolor=brown,      
 urlcolor=gold(metallic),
 citecolor=darkgoldenrod,
 }
\newtheorem{lthm}{Theorem}

\usepackage[all,cmtip]{xy}

\DeclareFontFamily{U}{wncy}{}
\DeclareFontShape{U}{wncy}{m}{n}{<->wncyr10}{}
\DeclareSymbolFont{mcy}{U}{wncy}{m}{n}
\DeclareMathSymbol{\Sh}{\mathord}{mcy}{"58}
\usepackage[T2A,T1]{fontenc}
\usepackage[OT2,T1]{fontenc}

\newtheorem{theorem}{Theorem}[section]
\newtheorem{lemma}[theorem]{Lemma}
\newtheorem{ass}[theorem]{Assumption}
\newtheorem*{theorem*}{Theorem}
\newtheorem*{ass*}{Assumption}
\newtheorem{definition}[theorem]{Definition}
\newtheorem{corollary}[theorem]{Corollary}
\newtheorem{remark}[theorem]{Remark}

\newtheorem{conjecture}[theorem]{Conjecture}
\newtheorem{proposition}[theorem]{Proposition}

\newcommand{\cF}{\mathcal{F}}
\newcommand{\cP}{\mathcal{P}}

\newcommand{\bT}{\mathbf{T}}

\newcommand{\Z}{\mathbb{Z}}

\newcommand{\Q}{\mathbb{Q}}

\newcommand{\cO}{\mathcal{O}}

\newcommand{\op}[1]{\operatorname{#1}}

\numberwithin{equation}{section}

\begin{document}

\title[Counting number fields whose Galois group is a wreath product]{Counting number fields whose Galois group is a wreath product of symmetric groups}

\author[H.~Mishra]{Hrishabh Mishra\, \orcidlink{0000-0003-3478-3261}}
\address[H.~Mishra]{Department of Mathematics
University of Wisconsin-Madison
480 Lincoln Drive, Madison, WI 53706, USA.}
\email{hmishra4@wisc.edu}

\author[A.~Ray]{Anwesh Ray\, \orcidlink{0000-0001-6946-1559}}
\address[A.~Ray]{Chennai Mathematical Institute, Chennai, India}
\email{anwesh@cmi.ac.in}


\maketitle

\begin{abstract}
  Let $K$ be a number field and $k\geq 2$ be an integer. Let $(n_1,n_2, \dots, n_k)$ be a vector with entries $n_i\in \Z_{\geq 2}$. Given a number field extension $L/K$, we denote by $\widetilde{L}$ the Galois closure of $L$ over $K$. We prove asymptotic lower bounds for the number of number field extensions $L/K$ with $[L:K]=\prod_{i=1}^k n_i$, such that $\op{Gal}(\widetilde{L}/K)$ is isomorphic to the iterated wreath product of symmetric groups $S_{n_1}\wr S_{n_2}\wr \dots \wr S_{n_k}$. Here, the number fields $L$ are ordered according to discriminant $|\Delta_L|:=|\op{Norm}_{K/\Q} (\mathcal{D}_{L/K})|$. The results in this paper are motivated by Malle's conjecture. When $n_1=n_2=\dots =n_k$, these wreath products arise naturally in the study of \emph{arboreal Galois representations} associated to rational functions over $K$. We prove our results by developing Galois theoretic techniques that have their origins in the study of dynamical systems.
\medskip\medskip

\noindent\textsc{R\'esum\'e.} Soient $K$ un corps de nombres et $k\geq 2$ un entier. Soit $(n_1,n_2,\dots,n_k)$ un $k$-uplet d'entiers tels que $n_i\in\mathbb{Z}_{\geq 2}$. Pour une extension de corps de nombres $L/K$, on note $\widetilde{L}$ la clôture galoisienne de $L$ sur $K$. Nous établissons des bornes asymptotiques inférieures pour le nombre d'extensions $L/K$ vérifiant \[ [L:K]=\prod_{i=1}^k n_i \] et telles que le groupe de Galois $\operatorname{Gal}(\widetilde{L}/K)$ soit isomorphe au produit en couronne itéré des groupes symétriques \[ S_{n_1}\wr S_{n_2}\wr \cdots \wr S_{n_k}. \] Les corps de nombres $L$ sont ordonnés suivant leur discriminant \[ |\Delta_L| = |\operatorname{Norm}_{K/\mathbb{Q}}(\mathcal{D}_{L/K})|. \] Nos résultats sont motivés par la conjecture de Malle. Lorsque $n_1=n_2=\cdots=n_k$, ces produits en couronne apparaissent naturellement dans l'étude des représentations galoisiennes arborées associées à des fonctions rationnelles sur $K$. Nous démontrons nos résultats en développant des techniques de théorie de Galois issues de l'étude des systèmes dynamiques.\end{abstract}

\section{Introduction}
\subsection{Background and motivation}
\par In this paper we prove asymptotic lower bounds for the number of
number field extensions with prescribed Galois group. The Galois groups considered here are wreath
products of symmetric groups and are natural to consider since they arise as
Galois groups of splitting fields of iterates of polynomials defined over number fields. In the
study of dynamical systems, such Galois extensions naturally arise, and it is convenient to leverage techniques from arithmetic dynamics.
\par Let $K$ be an algebraic number field with $d=[K:\Q]$. Given a number field extension $L/K$, let $\widetilde{L}$ denote the Galois closure or $L$ over $K$. Let $\mathcal{D}_{L/K}$ denote the relative discriminant, and let $N_{K/\Q}:K\rightarrow \Q$ denote the norm map. We set $\Delta_L:=N_{K/\Q}(\mathcal{D}_{L/K})$. Let $G$ be a finite transitive subgroup of $S_n$ and $L/K$ be an extension with $[L:K]=n$. Then, $\op{Gal}(\widetilde{L}/K)$ acts on the set of embeddings of $L$ into $\bar{K}$. For $X\in \mathbb{R}_{>0}$, consider the function 
\[N_{n,K}(X;G):= \#\{L/K\mid L\subset \bar{K}, [L:K]=n,
\op{Gal}(\widetilde{L}/K)\simeq G, |\Delta_L|\leq X\}.\]
Here, the isomorphism $\op{Gal}(\widetilde{L}/K)\simeq G$ is that of permutation subgroups of $S_n$. Malle \cite{malle2002distribution} made a precise conjecture regarding the asymptotic growth of $N_{n, K}(X;G)$ as $X$ goes to infinity. In greater detail, the conjecture predicts that \[N_{n,K}(X;G)\sim c(K, G)X^{a(G)}(\log X)^{b(K,G)-1},\]where $a(G)$ is a constant that only depends on $G$ and its permutation representation, and $b(K, G), c(K,G)$ are constants that depend both on $K$ as well as on $G$. This precise asymptotic prediction is known as the strong form of Malle's conjecture and has been proven for various groups $G\subseteq S_n$. 
\par We note that proving Malle's conjecture in generality is very difficult since a positive answer would imply the inverse Galois problem. The conjecture has been proven in special cases. For instance, the conjecture has been resolved for abelian groups by Maki \cite{maki1985density} and Wright \cite{wright1989distribution}, $S_n$ for $n\leq 5$ \cite{davenport1971density, bhargava2005density,bhargava2008density, bhargava2010density}, the dihedral group $D_4$ \cite{cohen2002enumerating}, as well as finite nilpotent groups satisfying additional conditions \cite{koymans2023malle}. For a more detailed and exhaustive list of results, we refer to the discussion in Section \ref{s 3.1}. The precise asymptotic predicted by the strong form of Malle's conjecture has been shown to not hold. Kl\"uners provided an explicit counterexample to the conjecture, cf. \cite{kluners2005counter}. There is however, a weak form of the conjecture, which is still widely expected to hold for all permutation groups $G$. Let us briefly state this conjecture. For a permutation $g\in S_n$, we define its \emph{index} as follows
\[\op{ind}(g):=n-\text{ the number of orbits of }g\text{ on }[n].\] Given a conjugacy class $C$ of $G$, let $\op{ind}(C)$ denote $\op{ind}(g)$, where $g\in C$. For any permutation group $G\neq 1$, set $G^\#:=G\backslash\{1\}$, and set
\[a(G):=\left(\op{min}\{\op{ind}(g)\mid g\in G^\#\}\right)^{-1}.\] 
\begin{conjecture}[Malle's conjecture -- weak form]
    Let $G$ be a transitive permutation group and $K$ be a number field. Then, for all $\epsilon>0$, there exist constants $c_1(K,G), c_2(K,G; \epsilon)>0$ such that 
    \[c_1(K,G)X^{a(G)}< N_{n,K}(X;G)<c_2(K,G;\epsilon)X^{a(G)+\epsilon},\]
    for all large enough values of $X$.
\end{conjecture}
Motivated by such developments, asymptotic lower bounds for $N_{n, K}(X;G)$ for various families of groups $G\subset S_n$ have been proven by various authors, and are listed below.
\begin{enumerate}
    \item When the inverse Galois problem is solved for $G$ over $K$, $n=|G|$, and $\iota: G\hookrightarrow S_n$ is the regular representation, a general asymptotic lower bound for $N_{n,K}(X;G)$ is proven by Kl\"uners and Malle, cf. \cite[Theorem 4.1]{klunersmalle}.
    \item For the identity $\iota:S_n\rightarrow S_n$, Malle's conjecture predicts that $N_n(S_n;G)\sim c_n X$, for some constant $c_n>0$ which depends only on $n$. Malle \cite{malle2002distribution} showed that $N_{n, \Q}(X, S_n)\gg X^{1/n}$.
    \item Ellenberg and Venkatesh \cite{ellenberg2006number} showed that $N_{n, K}(X;S_n)\gg X^{\frac{1}{2}+\frac{1}{n^2}}$. 
    \item Bhargava, Shankar and Wang \cite{bhargava2022squarefree} showed that $N_{n, \Q}(X;S_n)\gg X^{\frac{1}{2}+\frac{1}{n}}$.
    \item For the natural inclusion $\iota: A_n\hookrightarrow S_n$, Pierce, Turnage-Butterbaugh and Wood \cite{pierce2020effective} showed that $N_{n, \Q}(X;A_n)\gg X^{\frac{n!-2}{n!(4n-2)}}$. For $n\geq 6$ and $n\neq 7$, this result is further improved upon by Landesman, Lemke Oliver and Thorne \cite{landesman2021improved}, who have proven that 
    \[N_{n, K}(X;A_n)\gg \begin{cases}
        X^{\frac{(n-4)(n^2-4)}{8(n^3-n^2)}} & \text{ if }n\text{ is even,}\\
        X^{\frac{(n-7)(n+2)}{8n^2}} & \text{ if }n\text{ is odd.}\\
    \end{cases}\]\end{enumerate}
\subsection{Main result} We state our main result. Let $G_1\subseteq S_m$ and $G_2\subseteq S_n$ be two transitive permutation subgroups. The \emph{wreath product} $G=G_1\wr G_2$ is the semidirect product $G_1^n\rtimes G_2$, where $G_2$ permutes the $n$ copies of $G_1$. The group $G$ is seen to be a permutation subgroup of $S_{mn}$. For $\vec{n}=(n_1, \dots, n_k)\in \Z_{\geq 2}^k$, set 
    \[S(\vec{n}):=S_{n_1}\wr S_{n_2}\wr \dots \wr S_{n_k}.\] We remark that taking wreath products is an associative operation. Note that $S(\vec{n})$ is a subgroup of $S_N$, where $N:=\prod_{i=1}^k n_i$. Let $[S_n]^k$ denote the $k$-fold wreath power $S(\underbrace{n, n, n, \dots, n}_{k\text{-times}})$. These wreath powers are known to arise naturally in the study of arboreal Galois representations. For further details, we refer to Section \ref{section on arboreal}.
    \par We note that the permutation group $S(\vec{n})$ contains a transposition, and hence
\[
a\left(S(\vec{n})\right)=1.
\]
Therefore, the weak form of Malle's conjecture predicts that
\[
N_{N,K}(X;S(\vec{n}))\gg X
\]
and
\[
N_{N,K}(X;S(\vec{n}))\ll_{\varepsilon} X^{1+\varepsilon}
\]
for every $\varepsilon>0$.
    Below is the main result of the article, and is proven by refining the method of Ellenberg and Venkatesh \cite[Section 3]{ellenberg2006number}.
    \begin{lthm}\label{main thm}
    Let $k\geq 2$ and $\vec{n}=(n_1,\dots,n_k)\in \Z_{\geq 2}^k$. Denote by
    \[
    S(\vec{n})=S_{n_1}\wr S_{n_2}\wr \cdots \wr S_{n_k}
    \]
    the corresponding iterated wreath product of symmetric groups. Let $K$ be a number field. Set
    \[
    B=B(\vec{n})
    :=
    \sum_{j=1}^{k-1}
    \left(\frac{n_j-1}{2}\right)
    \left(\prod_{v=1}^j n_v\right)
    +
    \left(\frac{n_k+1}{2}\right)
    \left(\prod_{v=1}^k n_v\right),
    \]
    and
    \[
    N=N(\vec{n}):=\prod_{v=1}^k n_v.
    \]
    There is an absolute constant $A>0$ such that
    \[
    N_{N,K}\left(X;S(\vec{n})\right)\gg
    \begin{cases}
    X^{\frac{B-N/2}{N^2-N}}
    & \text{if } B\geq N\left(\op{exp}(A\sqrt{\log N})+\frac12\right),\\[6pt]
    X^{\frac{2\op{exp}(A\sqrt{\log N})(B-N)}
    {(2\op{exp}(A\sqrt{\log N})-1)(N^2-N)}}
    & \text{if } B\leq N\left(\op{exp}(A\sqrt{\log N})+\frac12\right).
    \end{cases}
    \]
\end{lthm}

The constant $A>0$ above is the absolute constant appearing in the
Ellenberg--Venkatesh upper bound
\[
N_{N,K}(X)\ll_{N,K}X^{\op{exp}(A\sqrt{\log N})}.
\]
This is used together with \cite[Proposition~2.7]{landesman2021improved}.
Over $K=\Q$, one can sharpen the case decomposition by using the improved
upper bound of Lemke Oliver and Thorne \cite{lemke2022upper}; in that case,
the quantity $\op{exp}(A\sqrt{\log N})$ above may be replaced by
$c(\log N)^2$ for a suitable absolute constant $c>0$. 
\par In order to get a better feeling for the bounds above, we specialize to the wreath powers $[S_n]^k$, in which case $N=n^k$.
\begin{lthm}\label{main corollary}
    Let $k,n\in \Z_{\geq 2}$ and let $[S_n]^k=S(n,n,\dots,n)$ denote the
    $k$-fold wreath power of $S_n$. Set
    \[
    E_{n,k}:=\op{exp}\left(A\sqrt{k\log n}\right),
    \]
    where $A>0$ is the absolute constant appearing in Theorem \ref{main thm}.
    Then
    \[
    N_{n^k,K}(X;[S_n]^k)\gg
    \begin{cases}
    X^{\delta_{n,k}}
    & \quad\text{if }\quad
    \displaystyle
    \frac{n+1-n^{1-k}}{2}\geq E_{n,k},\\[8pt]
    X^{\delta'_{n,k}}
    & \quad\text{if }\quad
    \displaystyle
    \frac{n+1-n^{1-k}}{2}\leq E_{n,k},
    \end{cases}
    \]
    where
    \[
    \delta_{n,k}
    :=
    \frac{n^{k+1}+n^k-n}{2(n^{2k}-n^k)}
    =
    \frac{n+1-n^{1-k}}{2(n^k-1)}
    \]
    and
    \[
    \delta'_{n,k}
    :=
    \frac{E_{n,k}(n^{k+1}-n)}
    {(2E_{n,k}-1)(n^{2k}-n^k)}
    =
    \frac{E_{n,k}}{2E_{n,k}-1}\cdot \frac{1}{n^{k-1}}.
    \]
    In particular, in both cases the exponent is at least $1/(2n^{k-1})$.
    Hence
    \[
    N_{n^k,K}(X;[S_n]^k)
    \gg
    X^{\frac{1}{2n^{k-1}}}.
    \]
\end{lthm}

\subsection{Organization} Including the introduction, the article consists of a total of $6$ sections. In Section \ref{s 2}, we set up notation that is used throughout the article. In Section \ref{s 3}, we discuss Malle's conjecture and various generalities on arboreal Galois representations associated to a polynomial function. In greater detail, we first state the precise form of Malle's conjecture and discuss some of the known results in this area. Then, the precise form of Malle's conjecture for iterated wreath products of symmetric groups is discussed. We discuss the tree structure associated to the roots of iterates of a polynomial function, and the associated Galois representation on this tree. We discuss seminal results of Odoni, which shall be used in establishing the results in this article. In Section \ref{s 4}, we discuss various number field counting techniques. We begin the section by discussing a certain refinement of the Hilbert irreducibility theorem due to Cohen \cite{cohen1981} and Landesman, Lemke Oliver and Thorne \cite{landesman2021improved}. We then outline the strategy of Ellenberg and Venkatesh in proving an asymptotic lower bound for $N_{n, K}(X;S_n)$. We then discuss some generalizations of this method, which can be applied to various subgroups $G\subset S_n$. In particular, we state a criterion due to Pierce, Turnage-Butterbaugh and Wood (cf. Theorem \ref{ptbw}). This criterion does apply to $K=\Q$, and it can be applied to give a bound that is weaker than that of Theorem \ref{main thm}. We refer to Remark \ref{only remark} for further details. Our method relies on a different strategy and we are able to obtain stronger results. Our key contributions are contained in Section \ref{s 5}, in which we make the final preparations for the proof of the Theorems \ref{main thm} and \ref{main corollary}. In this section, we apply results of Odoni to construct a polynomial over a function field over $K$, whose Galois group is $S(\vec{n})$. We then extend and generalize the strategy of Ellenberg and Venkatesh, outlined in Section \ref{EV section 4.2}, to prove our main result. These proofs of the Theorems \ref{main thm} and \ref{main corollary} are given in Section \ref{s 6}.

\subsection{Outlook} We expect that the methods introduced in this article can be sufficiently generalized to prove similar results for various wreath products of alternating and symmetric groups, i.e., to groups of the form $G_1\wr G_2 \wr \dots \wr G_k$, where $G_k$ is either a symmetric group, or an alternating group. Similar investigations would potentially lead to many new directions in Galois theory, arithmetic statistics and arithmetic dynamics. It is thus not only the result itself, but the method of proof that is of significance, as it leads to new questions as well as the prospect for further refinement.

\section*{Statements and Declarations}

\subsection*{Conflict of interest} The authors report that there is no conflict of interest to declare.

\subsection*{Data Availability} There is no data associated to the results of this manuscript.  

\subsection*{Acknowledgments} When the project was completed, the second named author's research was supported by the CRM-Simons postdoctoral fellowship. The authors thank the referee for the careful reading of the manuscript and for the helpful comments and suggestions. We are especially grateful for the referee's suggestion to use the upper bound of Ellenberg and Venkatesh in place of Schmidt's upper bound in the proofs of Theorems A and B. This suggestion led to a refinement of the statements of both theorems.

\section{Notation}\label{s 2}
This short section is devoted to setting up basic notation.
\begin{itemize}
\item Let $a<b$, we shall denote the set of integers $m$ that lie in the range $a\leq m \leq b$ by $[a,b]$.
    \item Let $K$ be a number field, with $d:=[K:\Q]$. We denote by $\cO_K$ its ring of integers.
    \item Let $\mathbb{A}^r$ denote the $r$-dimensional affine space, and \[\mathbb{A}^r(\cO_K)=\{\alpha=(\alpha_1, \dots, \alpha_r)\mid \alpha_i\in \cO_K, \forall i\}.\]
    \item Fix an algebraic closure $\bar{K}$ of $K$ and set $\op{G}_K:=\op{Gal}(\bar{K}/K)$.
    \item Given a number field extension $L/K$, let $\mathcal{D}_{L/K}$ denote the relative discriminant. Let $N_{K/\Q}:K\rightarrow \Q$ denote the norm map, and set $\Delta_L:=N_{K/\Q}(\mathcal{D}_{L/K})$.
    \item For $n\geq \Z_{\geq 2}$, let $[n]$ be the set of numbers $\{1,2, \dots, n\}$. The symmetric group of permutations of $[n]$ shall be denoted by $S_n$.
    \item Let $X>0$ be a real variable. Given positive functions $f(X)$ and $g(X)$, we write $f(X)\sim g(X)$ to mean that 
    \[\lim_{X\rightarrow \infty}\frac{f(X)}{g(X)}=1.\]
    \item We write $f(X)\gg g(X)$, if there is a constant $C>0$, such that $f(X)\geq C g(X)$ for all large enough values of $X$. In this article, the function $f(X)$ shall be defined for a given pair $(K, n)$, where $K$ is a number field and $n\in \Z_{\geq 2}$. The implied constant $C$ shall depend on both $K$ and $n$. 
    \item We shall write $f(X)\ll g(X)$ (or $f(X)=O\left(g(X)\right)$) to mean that there is a constant $C>0$ such that $f(X)\leq C g(X)$. We write $f(X)\ll_{n, K} g(X)$ when the constant $C$ may depend on $n$ and $K$.
    \item Let $L$ be a field and $f(X)$ be a non-zero polynomial with coefficients in $L$. Let $L_f$ denote the splitting field of $f(X)$ over $L$. The Galois group $\op{Gal}(L_f/L)$ is denoted $\op{Gal}(f(X)/L)$ and is referred to as the Galois group over $L$ generated by $f(X)$.
    \item For a number field $L/K$, let $\widetilde{L}$ be the Galois closure of $L$ over $K$.
    \item Let $\chi:\op{Gal}(\bar{\Q}/\Q)\rightarrow \hat{\Z}^\times$ denote the cyclotomic character. 
    \item For $\alpha\in K$, set $\lVert \alpha \rVert $ to denote the maximum archimedean valuation of $\alpha$.
    \item Let $f(x)=x^n +a_{1} x^{n-1}+\dots +a_n$ be a monic polynomial with coefficients $a_i\in \cO_K$. The \emph{height} of $f$ is defined as follows \[\lVert f\rVert:=\op{max}\{\lVert a_i\rVert^{\frac{1}{i}}\mid i\in [1, n]\}.\] 
    \item Throughout, $C$ or $C_i$ shall refer to a positive constant that depends on $K$ and $\vec{n}=(n_1, \dots, n_k)$.
\end{itemize}
\section{Arboreal representations and wreath products}\label{s 3}
\par We begin by discussing Malle's conjecture and the notion of a wreath product, as well as generalities on arboreal Galois representations.
\subsection{Malle's conjecture}\label{s 3.1}
\par Let $K$ be a number field with ring of integers $\cO_K$, and set $d:=[K:\Q]$. Let $G$ be a finite transitive subgroup of $S_n$. We shall refer to such a group $G$ as a \emph{permutation group}. Let $L/K$ be a number field extension with $[L:K]=n$ and let $\widetilde{L}$ be the Galois closure of $L$ over $K$. Then, $\op{Gal}(\widetilde{L}/K)$ acts on the set of embeddings of $L$ into $\bar{K}$. Two permutation subgroups $G_1$ and $G_2$ of $S_n$ are isomorphic as permutation groups if there is an isomorphism $G_1\xrightarrow{\sim} G_2$ induced by a reordering of the set $[n]$. Given a finite group $G$, an injective homomorphism $\iota: G\hookrightarrow S_n$ is referred to as a permutation representation of $G$.
\par Let $\iota: G\hookrightarrow S_n$ be a permutation representation. For $X\in \mathbb{R}_{>0}$, set
\[\begin{split}
N_{n,K}(X):=& \#\{L/K\mid L\subset \bar{K}, [L:K]=n, |\Delta_L|\leq X\}, \\
N_{n,K}(X;G, \iota):=& \#\{L/K\mid L\subset \bar{K}, [L:K]=n,
\op{Gal}(\widetilde{L}/K)\simeq G, |\Delta_L|\leq X\}, \\ \end{split}\] where the isomorphism of $\op{Gal}(\widetilde{L}/K)$ with $G$ is that of permutation groups. When it is clear from the context, we suppress the dependence of $N_{n, K}(X;G, \iota)$ on the permutation representation $\iota$ and simply write $N_{n,K}(X;G):=N_{n,K}(X;G, \iota)$. It is clear that 
\[N_{n,K}(X)=\sum_{\iota} N_{n,K}(X; G, \iota),\] where $\iota$ ranges over all isomorphism classes of permutation representations $\iota:G\hookrightarrow S_n$.

\par It is expected that $N_{n, K}(X)\sim c_{n, K} X$, where $c_{n, K}$ is positive constant which depends only on $n$ and $K$, cf. \cite[p.723]{ellenberg2006number}. Asymptotic upper bounds for $N_{n, K}(X)$ have been established in various works.
\begin{enumerate}
    \item Schmidt \cite{schmidt1995number} showed that $N_{n, K}(X)\ll_{n, K} X^{(n+2)/4}$.
    \item Ellenberg and Venkatesh \cite{ellenberg2006number} improved the above, and showed that \[N_{n, K}(X)\ll_{n, K} X^{\op{exp}(C\sqrt{\log n})},\] where $C>0$ is an absolute constant.
    \item Couveignes \cite{couveignes2020enumerating} showed that $N_{n, \Q}(X)\ll_{n} X^{c(\log n)^3}$, where $c>0$ is an (unspecified) absolute constant.
    \item Lemke Oliver and Thorne \cite{lemke2022upper} improved upon the above result and showed that $N_{n,\Q}(X)\ll_{n} X^{c(\log n)^2}$ for a certain constant $c>0$. If one takes $c:=1.564$, the result holds for all sufficiently large values of $n$.
\end{enumerate}
When $\iota: G\hookrightarrow S_n$ is a fixed embedding, the conjectured asymptotic for $N_n(X;G)$ is due to Malle. For a permutation $g\in S_n$, we define its \emph{index} as follows
\[\op{ind}(g):=n-\text{ the number of orbits of }g\text{ on }[n].\] Given a conjugacy class $C$ of $G$, let $\op{ind}(C)$ denote $\op{ind}(g)$, where $g\in C$. For any group $G\neq 1$, set $G^\#:=G\backslash\{1\}$, and set
\[a(G):=\left(\op{min}\{\op{ind}(g)\mid g\in G^\#\}\right)^{-1}.\] We note that $a(G)$ depends not only on $G$, but also on the permutation representation $G\hookrightarrow S_n$. We note that if $G$ contains a transposition, then, $a(G)=1$. In particular, $a(S_n)=1$. It is easy to show that $a(A_n)=\frac{1}{2}$, where $A_n\subset S_n$ is the alternating subgroup.

\par The Galois group $\op{G}_{K}:=\op{Gal}(\bar{K}/K)$ acts on the set of conjugacy classes of $G$ via $\sigma \cdot C:=C^{\chi(\sigma)}$, where $C$ is a conjugacy class and $\sigma\in \op{G}_{K}$. With respect to notation above, set 
\[b(G,K):=\# \left(\{C\in \op{Cl}(G)\mid \op{ind}(C)=a(G)^{-1}\}/\op{G}_{K}\right).\]The conjecture stated below is often referred to as the weak form of Malle's conjecture \cite[p.316]{malle2002distribution}. 
\begin{conjecture}[Malle's conjecture -- weak form]
    Let $G$ be a transitive permutation group and $K$ be a number field. Then, for all $\epsilon>0$, there exist constants $c_1(K,G), c_2(K,G; \epsilon)>0$ such that 
    \[c_1(K,G)X^{a(G)}\leq N_{n,K}(X;G)<c_2(K,G;\epsilon)X^{a(G)+\epsilon},\]
    for all large enough values of $X$.
\end{conjecture}
\par Heuristics supporting the above conjecture are discussed in \cite[Section 7]{malle2002distribution}. In \cite{malle2004distribution}, Malle states the following stronger form of the conjecture, predicting the precise asymptotic growth of $N_n(X)$.
\begin{conjecture}[Malle's conjecture -- strong form]\label{Malle strong form}
    Let $n\geq 2$ be an integer and $G$ be a transitive permutation subgroup of $S_n$. Then, with respect to notation above, 
    \[N_{n,K}(X;G)\sim c(K, G)X^{a(G)}(\log X)^{b(K,G)-1},\] where $c(K,G)>0$ is a constant which depends on $K$ and the permutation group $G$.
\end{conjecture}
The strong form of the conjecture is known in various cases, some of which are listed below.
\begin{enumerate}
    \item When $G$ is abelian and $\iota: G\hookrightarrow S_{|G|}$ be the regular representation, Malle's conjecture is resolved by Maki \cite{maki1985density} and Wright \cite{wright1989distribution}.
    \item When $G=S_n$ for $n\leq 5$ and $\iota: S_n\rightarrow S_n$ is the identity, the conjecture is resolved by Davenport, Bhargava, cf. \cite{davenport1971density, bhargava2005density,bhargava2008density, bhargava2010density}.
    \item For $D_4\subset S_4$, the conjecture is resolved by Cohen, Diaz y Diaz and Olivier \cite{cohen2002enumerating}.
    \item Groups of the form $S_n\times A$, where $n=3,4,5$ and $|A|$ is coprime to
$2,6,30$ respectively, cf. \cite{wang2021malle}. We also refer the reader to the
subsequent refinement of Masri, Thorne, Tsai, and Wang \cite{masri2020malle}.
    \item Let $G$ be a nontrivial finite nilpotent group and $p$ be the smallest prime that divides $\# G$. Assume that all elements in $G$ that are of order $p$ are central. Then, under these hypotheses, Koymans and Pagano prove the strong form of the Malle conjecture for the regular representation of $G$, cf. \cite{koymans2023malle}.
\end{enumerate}
The above list is not exhaustive. The strong form is however now known to be false. Kl\"uners \cite{kluners2005counter} showed that for the group $G=C_3\wr C_2$, the predicted factor of $\log X$ is too small. The weak version of Malle's conjecture is however expected to be true.

  \subsection{Wreath products}\label{s 3.2} 
\par Let $G_1\subseteq S_m$ and $G_2\subseteq S_n$ be two transitive permutation subgroups. The \emph{wreath product} $G=G_1\wr G_2$ is the semidirect product $G_1^n\rtimes G_2$, where $G_2$ permutes the $n$ copies of $G_1$. We write $g=g_1\rtimes g_2$, where $g_1=(\sigma_1, \dots, \sigma_n)\in G_1^n$. The group $G$ acts on $[m]\times [n]=\{(i,j)\mid i\in [m], j\in [n]\}$ as follows
\[g\cdot (i, j):=(\sigma_j(i), g_2(j)).\]We thus realize $G$ is a subgroup of $S_{mn}$. For $k\in \Z_{\geq 1}$ we set $[S_n]^k$ to denote the $k$-fold wreath product of $S_n$. In greater detail, for $k=1$, we set $[S_n]^1:=S_n$ and $[S_n]^k:=S_n\wr [S_n]^{k-1}$ for $k\geq 2$. Implicit to the construction, we have a natural permutation representation, $[S_n]^k\hookrightarrow S_{n^k}$. For $k>d$, there is a natural quotient map $[S_n]^k\rightarrow [S_n]^d$. We let $[S_n]^{\infty}$ denote the inverse limit $\varprojlim_{k} [S_n]^k$. 
\begin{proposition}
    With respect to notation above, we have that $a([S_n]^k)=a(S_n)=1$. 
\end{proposition}
\begin{proof}
    %According to \cite[Lemma 5.1]{malle2002distribution}, $a(G_1\wr G_2)=a(G_1)$. In particular, \[a([S_n]^k)=a(S_n \wr [S_n]^{k-1})=a(S_n).\] 
    It is easy to see that if $g$ is a transposition, then, $\op{ind}(g)=1$. Since $[S_n]^k$ contains a transposition, this implies that $a([S_n]^k)=1$. 
\end{proof}

Therefore, the weak version of Malle's conjecture predicts that 
\[c_1 X\leq N_{n,K}(X;[S_n]^k)<c_2(\epsilon)X^{1+\epsilon},\] for any value of $\epsilon>0$. 
\subsection{Images of arboreal Galois representations}\label{section on arboreal}

\par We briefly introduce the theory arboreal representations, for further details, we refer to \cite{jones2013galois}. We also discuss a result of Odoni \cite{odoni1985galois}, which shall prove to be crucial in establishing our main results in later sections. In this section, we let $K$ be a field of characteristic $0$, with algebraic closure $\bar{K}$, and $f(x)\in K[x]$ be a polynomial of degree $n\geq 2$. For $k\in \Z_{\geq 1}$, let $f^{\circ k}:=f\circ f \circ \dots \circ f$ be the $k$-th iterate of $f$. Thus, $f^{\circ 1}=f$, $f^{\circ 2}=f\circ f$,  $f^{\circ 3}=f\circ f\circ f$ and so on. We let $f^{\circ 0}(x)=x$ denote the identity polynomial. %We view these iterates as functions on $\bar{K}$, and thus the family $\{f^{\circ k}\mid k\in \Z_{\geq 1}\}$ gives rise to a dynamical system on $\bar{K}$. For $\alpha\in \bar{K}$, refer to the set of translates $\cO_f(\alpha):=\{f^{\circ k}(\alpha)\mid k\in \Z_{\geq 0}\}$ as the \emph{orbit of $\alpha$}. The point $\alpha$ is said to be \emph{pre-periodic} (resp. \emph{wandering}) if $\cO_f(\alpha)$ is finite (resp. infinite). Let $t\in K$ be an arbitrary element. We introduce an assumption which is shown to be satisfied for various examples of interest.

\begin{ass}\label{sep ass}
    Let $t\in K$ and $f\in K[x]$ be a polynomial. Assume that for all $k\in \Z_{\geq 1}$, the polynomial $f^{\circ k}(x)-t\in K[x]$ has no repeated roots.
\end{ass}

For the rest of this section, we shall assume that the Assumption \ref{sep ass} is satisfied for the pair $(f, t)$. %Note that in this case, $t$ must be a wandering point. 
Thus, for all $k\in \Z_{\geq 0}$, the preimage set 
\[f^{-k}(t):=(f^{\circ k})^{-1}(t)=\{z\in \bar{K}\mid f^{\circ k}(z)=t\}\] has cardinality equal to $n^k$. For $k=0,1,2$, we find that
\[\begin{split}
    &f^{-0}(t)=\{t\},\\
    &f^{-1}(t)=\{z\in \bar{K}\mid f(z)=t\},\\
    &f^{-2}(t)=\{z\in \bar{K}\mid f(f(z))=t\},\dots, \text{etc.}
\end{split}\]We identify the splitting field $K_{f^{\circ k}-t}$ with $K\left(f^{-k}(t)\right)$. %Since $t$ is a wandering point, $f^{-k}(t)$ and $f^{-m}(t)$ are disjoint unless $k=m$. 
We denote by $T_k(f, t)$ the union $\bigcup_{j\leq k} f^{-j}(t)$, and $T_\infty(f, t)$ is defined as the infinite union $\bigcup_{j=1}^\infty f^{-j}(t)$. The sets $T_k(f,t)$ and $T_\infty(f,t)$ have a natural tree structure. The vertices consist of the elements in $T_k(f, t)$ (resp. $T_\infty(f,t)$) and the vertices $\alpha$ and $\beta$ are connected by an edge if $f(\alpha)=\beta$. For an explicit example, we refer to \cite[p. 417]{kadets2020large}. 

\par Let $\bT_k$ be the perfect $n$-ary rooted tree with height $k$. This tree has a single root, every node of height less than $k$ has $n$ child nodes and all leaves are of height $k$. We view $\bT_k$ as a subgraph of $\bT_{k+1}$, and we let $\bT_\infty$ be the $n$-regular tree with infinite height. We identify $\bT_\infty$ with the infinite union $\bigcup_{k\geq 1} \bT_k$. We have identifications of $\bT_k$ with $T_k(f,t)$ and $\bT_\infty$ with $T_\infty(f,t)$. The point $\{t\}$ is the root of $T_k(f,t)$, and the vertices in $f^{-k}(t)$ are at height $k$. The automorphism group $\op{Aut}(\bT_k)$ (resp. $\op{Aut}(\bT_\infty)$) is isomorphic to the wreath product $[S_n]^k$ (resp. $[S_n]^\infty$).
\par The Galois group $\op{G}_K:=\op{Gal}(\bar{K}/K)$ acts on $\bT_k=T_k(f,t)$ and $\bT_\infty=T_\infty(f,t)$ via automorphisms, and let \[
\begin{split}& \rho_{f,t,k}:\op{G}_K\rightarrow \op{Aut}(\bT_k)\xrightarrow{\sim} [S_n]^k,\\
& \rho_{f,t,\infty}:\op{G}_K\rightarrow \op{Aut}(\bT_\infty)\xrightarrow{\sim} [S_n]^\infty,\\
\end{split}\]
be the associated Galois representations. We refer to $\rho_{f,t,\infty}$ as the \emph{arboreal Galois representation} associated to $(f,t)$, and set \[\rho_{f, k}:=\rho_{f,0, k}\text{ and }\rho_{f, \infty}:=\rho_{f,0, \infty}\] for ease of notation. Let $G_k(f,t)$ (resp. $G_\infty(f,t)$) denote the image of $\rho_{f,t, k}$ (resp. $\rho_{f,t,\infty}$). Note that $G_\infty(f,t)$ is identified with the inverse limit $\varprojlim_k G_k(f,t)$.

%\begin{question}
    %Let $K$ be a Hilbertian field of characteristic $0$ (cf. \cite{bary2008characterization} for a characterization of such fields). 
    %\begin{enumerate}
        %\item For which pairs $(f,t)$ does one have that $[\op{Aut}(\bT_\infty):G_\infty(f,t)]<\infty$?
        %\item For which pairs is $\rho_{f,t,\infty}$ surjective?
    %\end{enumerate}
%\end{question}

%Motivated by the above question,
Odoni proved an important result about the surjectivity of arboreal Galois representations in a general setting. Let $K$ be any field of characteristic $0$ (not necessarily Hilbertian), let $t_0, \dots, t_{n-1}$ be algebraically independent over $K$ and $L$ denote the function field $K(t_0, \dots, t_{n-1})$. In accordance with notation defined above, $G_\infty(f)$ is the image of the arboreal representation 
\[\rho_{f, \infty}:\op{G}_L\rightarrow \op{Aut}(\bT_\infty).\]
\begin{theorem}[Odoni]\label{odoni's theorem}
    Let $K$ be a field of characteristic $0$. With respect to notation above, set $f:=x^n+t_{n-1} x^{n-1}+\dots+ t_1 x +t_0$ and $L:=K(t_0, \dots, t_{n-1})$. Then, the following assertions hold
    \begin{enumerate}
        \item the Assumption \ref{sep ass} is satisfied for $(f,0)$,
        \item the representation
    \[\rho_{f, \infty}:\op{G}_L\rightarrow \op{Aut}(\bT_\infty)\] associated to $f$ and its iterates is surjective. In other words, \[\op{Gal}\left(L\left(f^{-k}(0)\right)/L\right)\simeq [S_n]^k\] for all $k\in \Z_{\geq 1}$.
    \end{enumerate}
\end{theorem}
\begin{proof}
    The first part follows from \cite[Lemma 2.3]{odoni1985galois}, and the second part is \cite[Theorem 1]{odoni1985galois}.
\end{proof}
%The above result has the following interesting consequence.

%\begin{corollary}\label{cor 3.7}
    %Let $K$ be a Hilbertian field of characteristic $0$. Then, for any $k\in \Z_{\geq 1}$ and $n\in \Z_{\geq 2}$, there are infinitely many polynomials $f(x)\in K[x]$ of degree $n$, such that $\op{Gal}\left(K\left(f^{-k}(0)\right)/K\right)$ is isomorphic to $[S_n]^k$.
%\end{corollary}

%\begin{proof}
    %The result follows from Theorem \ref{odoni's theorem} and the Hilbert irreducibility theorem.
%\end{proof}
%The above result motivates the following conjecture, cf. \cite[Conjecture 7.5]{odoni1985prime}.

%\begin{conjecture}[Odoni's conjecture]
    %Let $K$ be a Hilbertian field of characteristic $0$. Then, for any $n\in \Z_{\geq 2}$, there exists a polynomial $f(x)\in K[x]$ of degree $n$, such that $\op{Gal}\left(K\left(f^{-k}(0)\right)/K\right)$ is isomorphic to $[S_n]^k$ for all $k\in \Z_{\geq 1}$. In other words, there exists $f$ for which the associated arboreal representation $\rho_{f, \infty}$ is surjective.
%\end{conjecture}

%When $K$ is a number field and $n$ is even, or $[K:\Q]$ is odd, the above conjecture is proven by Benedetto and Juul \cite{benedetto2019odoni}. The proof of the conjecture for all values of $n$ can be found in a preprint of Specter \cite{specter2018polynomials}. Note that the conjecture does not hold for a general Hilbertian field, cf. \cite{odonifalse}.

\section{Number field counting techniques}\label{s 4}

\par Let $K$ be a number field and $d:=[K:\Q]$. In this section, we introduce the general strategy used in establishing asymptotic lower bounds for $N_{n,K}(X;G)$ for a transitive subgroup $G$ of $S_n$. We find it beneficial to give a brief account of the technique of Ellenberg and Venkatesh \cite{ellenberg2006number}, who show that $N_{n,K}(X;S_n)\gg X^{\frac{1}{2}+\frac{1}{n^2}}$. 
\subsection{Hilbert irreducibility}
\par In this section, we discuss Cohen's integral version of the Hilbert irreducibility theorem, cf. \cite{cohen1981}. We follow the notation and conventions from \cite[Appendix]{landesman2021improved}. Let $K$ be a number field, and given a subset $S\subset \mathbb{A}^r(\cO_K)=\cO_K^r$, and $e_1, \dots, e_r\in \mathbb{R}_{>0}$. We set
\[S(Y; e_1, \dots, e_r):=\left\{(\alpha_1, \dots, \alpha_r)\in S\mid \lVert \alpha_i\rVert \leq Y^{e_i}\right\}.\]
It is easy to see that for $S=\mathbb{A}^r(\cO_K)$, as $Y\rightarrow \infty$, 
\[\#\left(\mathbb{A}^r(\cO_K)\right)(Y; e_1, \dots, e_r)\sim c Y^{d\left(\sum_{i=1}^r e_i\right)},\] where $c>0$ is a constant which depends on $K$.
Let \[\op{Prob}_S(T;e_1, \dots, e_r):=\frac{\#S(Y; e_1, \dots, e_r)}{\# \left(\mathbb{A}^r(\cO_K)\right)(Y; e_1, \dots, e_r)}\]
be the probability that a point in $(\alpha_1, \dots, \alpha_r)\in \mathbb{A}^r(\cO_K)$ with coordinates $\lVert \alpha_i\rVert \leq Y^{e_i}$ is in $S$. Let $\mathbf{L}=K(a_1, \dots, a_r)$ be the function field over $K$ in variables $a_1, \dots, a_r$, and let $F(x)\in K[a_1, \dots, a_r][x]$ be a polynomial, such that $G=\op{Gal}(\mathbf{L}_F/\mathbf{L})$. For $\alpha=(\alpha_1, \dots, \alpha_r)\in K^r$, let $F_{\alpha}(x)\in K[x]$ be the specialization of $F(x)$ upon specialization $a_i\mapsto\alpha_i$ for $i=1, \dots, r$. Set $G_{\alpha}:=\op{Gal}(K_{F_{\alpha}}/K)$, and let $S$ be the set of vectors $\alpha\in \mathbb{A}^r(\cO_K)$ for which $G_{\alpha}=G$.
\begin{theorem}[Hilbert irreducibility]\label{Hilbert irred}
    With respect to notation above, we have that 
    \[\lim_{Y\rightarrow \infty} \op{Prob}_S(Y;e_1, \dots, e_r)=1.\]
\end{theorem}
\begin{proof}
    For the proof of the result, we refer to \cite[Theorem A.2]{landesman2021improved}.
\end{proof}

\par As an aside, we obtain a quantitative refinement of Corollary \ref{cor 3.7}, in the case when $K$ is a number field. For $n\in \Z_{\geq 2}$, let $\op{Poly}_n(\cO_K;Y)$ consist of all monic polynomials $P(x)$ of degree $n$ with coefficients in $\cO_K$ and $\lVert P\rVert\leq Y$. Any such polynomial 
\[P(x)=x^n+\alpha_{1}x^{n-1}+\alpha_2 x^{n-2}+\dots +\alpha_{n-1} x+\alpha_n\]is associated with a tuple \[\alpha=(\alpha_1, \dots, \alpha_n)\in \left(\mathbb{A}^n(\cO_K)\right)(Y; 1,2,3, \dots, n).\] Let $T_{n,k}(Y)$ be the subset of $P\in\op{Poly}_n(\cO_K;Y)$ for which \[\op{Gal}\left(K(P^{-k}(0))/K\right)\simeq [S_n]^k.\]
\begin{theorem}
    Let $n\in \Z_{\geq 2}$ and $k\in \Z_{\geq 1}$, then, 
    \[\lim_{Y\rightarrow \infty} \left(\frac{\# T_{n,k}(Y)}{\#\op{Poly}_n(\cO_K;Y)}\right)=1.\]
    In other words, the set of polynomials $P(x)$ of degree $n$ with coefficients in $\cO_K$, for which the representation
    \[\rho_{f,k}:\op{G}_K\rightarrow \op{Aut}(\mathbf{T}_k)\]is surjective, has density $1$.
\end{theorem}
\begin{proof}
   The result follows directly from Theorem \ref{odoni's theorem} and the Hilbert irreducibility theorem (cf. Theorem \ref{Hilbert irred}).
\end{proof}

\subsection{The method of Ellenberg and Venkatesh}\label{EV section 4.2}
\par Let us first briefly review the construction of Ellenberg and Venkatesh \cite{ellenberg2006number}, who show that $N_{K,n}(X;S_n)\gg X^{\frac{1}{2}+\frac{1}{n^2}}$. Given an extension $L/K$, set
\[\cO_{L}^0:=\{z\in \cO_L\mid \op{Tr}_K^L(z)=0\}.\]
Let $L/K$ be an extension for which $[L:K]=n$, and suppose that $z\in \cO_L^0$ is an element such that $K(z)=L$, then $z$ satisfies a polynomial of the form
\[f(x)=x^n+\alpha_2 x^{n-2}+\alpha_3x^{n-3}+\dots+ \alpha_{n-1} x+\alpha_n,\] with $a_i\in \cO_K$. Consider the polynomial 
\[F(x):=x^n+a_2 x^{n-1}+a_3x^{n-2}+\dots +a_{n-1}x+a_n\]
with coefficients in the polynomial ring $A:=K[a_2, a_3, \dots, a_{n-1}, a_n]$. Let $\mathbf{L}$ denote the function field $K(a_2, \dots, a_n)$,
and recall that $\mathbf{L}_F$ is the splitting field of $F$ over $\mathbf{L}$. Then, the Galois group $\op{Gal}(\mathbf{L}_F/\mathbf{L})$ is isomorphic to $S_n$ (cf. \emph{loc. cit.} for further details). Specializing the variables $a_i$ to values $\alpha_i\in \cO_K$, allows one to construct many extensions $L/K$ of degree $n$, for which $\op{Gal}(\widetilde{L}/K)\simeq S_n$. Let $Y>0$ be a real number and set 
\[\begin{split}& S(Y):=\{z\in \cO_{\bar{K}}\mid \op{Tr}_K^{K(z)}(z)=0, [K(z):K]=n, \lVert z\rVert\leq Y\}, \\ 
& S(Y;S_n):=\{z\in S(Y) \mid \op{Gal}\left(\widetilde{K(z)}/K\right)\simeq S_n\}. \\
\end{split}\]
In order to estimate $\# S(Y)$, one counts the total number of characteristic polynomials 
\[f(x)=x^n+\alpha_2 x^{n-2}+\alpha_3 x^{n-3}+\dots+\alpha_n\] for which $\lVert f\rVert \leq Y$. In other words, one counts the number of tuples $(\alpha_2, \dots, \alpha_n)\in \cO_{\bar{K}}^{(n-1)}$ such that $\lVert a_i \rVert \leq Y^i$ for all $i$. One finds that \[\# S(Y)\gg \#\left(\mathbb{A}^r(\cO_K)\right)(Y; 2,3, 4, \dots, (n-1), n)\sim  c Y^{\left(\frac{n(n+1)}{2}-1\right)d},\] and it follows from the Hilbert irreducibility theorem that the same asymptotic lower bound holds for $\# S(Y;S_n)$. Let $L/K$ be an extension for which $[L:K]=n$ and $\op{Gal}(\widetilde{L}/K)\simeq S_n$. A lower bound for $\# S(Y;S_n)$ gives rise to a lower bound for $N_{n, K}(X;S_n)$, provided one is able to estimate the number of $z\in S(Y;S_n)$ for which $L=K(z)$. It is shown that 
\[M_{L/K}(Y):=\#\{z\in S(Y;S_n)\mid K(z)\simeq L\}\ll \frac{Y^{(n-1)d}}{|\Delta_L|^{\frac{1}{n}}},\]cf. \cite[p. 738, ll.22-27]{ellenberg2006number} for further details. The lower bound for $\# S(Y;S_n)$ and the upper bound for $M_{L/K}(Y)$ together show that 
\[N_{n, K}(X)\gg X^{\frac{1}{2}+\frac{1}{n^2}},\] where $X:=Y^{n(n-1)d}$, we refer to \emph{loc. cit.} for further details.
\subsection{Generalizations to other groups $G\subset S_n$}
\par Let $G$ be a subgroup of $S_n$ for which the inverse Galois problem is solved over $K$. Moreover, suppose that there is a polynomial 
\[F(x)=x^n+\sum_{i=1}^n P_i(a_1, \dots, a_r) x^{n-i},\] with coefficients in the polynomial ring $A:=\cO_K[a_1, \dots, a_r]$, which cuts out a $G$-extension of $\mathbf{L}:=K(a_1, \dots, a_r)$. This is to say that $F(X)$ is irreducible over $\mathbf{L}$ and $\op{Gal}(\mathbf{L}_F/\mathbf{L})\simeq G$ as a permutation subgroup of $S_n$. Given a monomial $a_1^{b_1}a_2^{b_2}\dots a_r^{b_r}$, define the total degree to be the sum $\sum_{i=1}^r b_i$. The total degree of a polynomial is then defined to be the maximal degree of monomials in its support. Let $D$ denote the maximal total degree of the polynomials $P_i$ in the variables $\{a_i\mid i=1, \dots, r\}$. 

\begin{definition}
    Let $F(X)\in K[a_1, \dots, a_r][X]$ be a non-zero polynomial and $\mathbf{L}_F$ be its splitting field over $\mathbf{L}:=K(a_1, \dots, a_{r})$. Then, $F$ is said to be \emph{regular} if \[\mathbf{L}_F\cap \bar{K}(a_1, \dots, a_r)=\mathbf{L}.\]
\end{definition}
When $K=\Q$, and $P$ is regular, then, Pierce, Turnage--Butterbaugh and Wood \cite{pierce2020effective} establish a general asymptotic lower bound for $N_{n, \Q}(X;G)$.
\begin{theorem}[Pierce, Turnage-Butterbaugh, Wood]\label{ptbw}
    Let $n\in \Z_{\geq 2}$ and $G$ be a transitive subgroup of $S_n$. Suppose that $f\in \Q[X, a_1, \dots, a_r]$ is a regular polynomial with total degree $D$ in the $\{a_i\}$ and degree $n$ in $X$, with Galois group $G$ over $\Q(a_1, \dots, a_r)$. Then, for $X\geq 1$ and $\epsilon>0$, we have that 
    \[N_n(X;G)\gg_{f,\epsilon} X^{\frac{1-|G|^{-1}}{D(2n-2)}-\epsilon}.\]
\end{theorem}
\begin{proof}
    The above result it \cite[Theorem 1.5]{pierce2020effective}.
\end{proof}
The above result can be used in a large number of situations. All that is required is knowledge of the quantity $D$. On the other hand, the method of Ellenberg and Venkatesh requires more precise estimates. Applying the above result to $G=S_n$ yields $N_{n, \Q}(X; S_n)\gg_{n, \epsilon} X^{\frac{1-|n!|^{-1}}{(2n-2)}-\epsilon}$, which is weaker than the result of Ellenberg and Venkatesh, which asserts that $N_{n, K}(X)\gg X^{\frac{1}{2}+\frac{1}{n^2}}$. The exponent of $X$ is $O(\frac{1}{n})$, while that of Ellenberg and Venkatesh is $\frac{1}{2}+O(\frac{1}{n^2})$. This latter estimate is closer to Malle's conjecture, which predicts that $N_{n, K}(X;S_n)\sim a_{n, K} X$.

\par When the number of variables $r$ is suitably large, and it is possible to make precise estimates, it shall be more fruitful to generalize the method of Ellenberg and Venkatesh outlined in the previous section. We end the section by recalling an explicit criterion from \cite{landesman2021improved}, which shall prove to be especially useful. Let $G$ be a transitive subgroup of $S_n$ and $K/\Q$ be a number field and set $d:=[K:\Q]$. We set
\[\cF_{n, K}(X;G):=\{L/K\mid [L:K]=n, \op{Gal}(\widetilde{L}/K)\simeq G, |\Delta_L|\leq X\},\] and note that $N_{n, K}(X;G):=\# \cF_{n, K}(X;G)$. For $Y>0$, we let 
\[\cP_{n, K}(Y;G):=\{z\in \cO_{\bar{K}}\mid \lVert z\rVert \leq Y, [K(z):K]=n, \op{Gal}(\widetilde{K(z)}/K)\simeq G\}.\]
\begin{proposition}\label{prop 4.3}
    With respect to the above notation, assume that there is a constant
    $C>n$ such that
    \[
        \#\cP_{n,K}(Y;G)\gg Y^{dC}.
    \]
    Suppose further that, for some real number $e\geq 1/(n-1)$, one has an
    upper bound
    \[
        N_{n,K}(X;G)\ll X^e.
    \]
    Then
    \[
        N_{n,K}(X;G)\gg
        \begin{cases}
        X^{\frac{C-n/2}{n^2-n}}
        & \text{if } C\geq n\left(e+\frac12\right),\\[6pt]
        X^{\frac{2e(C-n)}{(2e-1)(n^2-n)}}
        & \text{otherwise.}
        \end{cases}
    \]
\end{proposition}
\begin{proof}
    The above result is \cite[Proposition 2.7]{landesman2021improved}.
\end{proof}

\section{Wreath products of symmetric groups}\label{s 5}
\par Let $k\in \Z_{\geq 2}$, and $\vec{n}=(n_1, \dots, n_k)\in \Z_{\geq 2}^k$. Throughout, we fix a number field $K$ and let $d:=[K:\Q]$. Set $N:=n_1 n_2 \dots n_k$ and let $S(\vec{n})\subset S_N$ be the wreath product $S_{n_1}\wr S_{n_2}\wr \dots \wr S_{n_k}$. In particular, the $k$-fold wreath product $[S_n]^k=S(n, n, \dots, n)\subset S_{n^k}$. We generalize the method of Ellenberg and Venkatesh to obtain an asymptotic lower bound for $N_{n, K}(X;S(\vec{n}))$. 
\begin{definition}Let $f(x)=\sum_{i=1}^d a_i(T_1, \dots, T_u) x^i$ and $g(x)=\sum_{j=1}^e b_j(S_1, \dots, S_v) x^j$, where $a_i(T_1, \dots, T_u)\in \cO_K[T_1, \dots, T_u]$ and $b_j(S_1, \dots, S_v)\in \cO_K[S_1, \dots, S_v]$. We define the polynomial $g\wr f\in \cO_K[T_1, \dots, T_u, S_1, \dots, S_v][x]$ as follows
\[g\wr f:=f(g(x))=\sum_{i=1}^d a_i(T_1, \dots, T_u) \left(\sum_{j=1}^e b_j(S_1, \dots, S_v) x^j\right)^i.\]
\end{definition}
Let $f_{u, n}(x)=x^n+\sum_{v=1}^{n} T_{u,v} x^{n-v}$ and set \begin{equation}\label{def of F}F:=f_{1,n_1}\wr f_{2, n_2}\wr \dots \wr f_{k, n_k}\in A[x],\end{equation} where \begin{equation}\label{defn of A}A:=\cO_K[\{T_{u,v}\mid u\in [1, k], v\in [1, n_i]\}].\end{equation}

\par For $i\in [1, k]$, we set $g_i:=f_{i, n_i}$, and for $j\in [1, k]$, set
\[F_j:=g_1\wr g_2\wr \dots \wr g_{j-1} \wr g_j\in A[x],\] and thus, $F=F_k$. We shall set $N_j:=\prod_{i=1}^j n_i$ and 
\[D_j:=\begin{cases}
    1 &\text{ if }j=1;\\
    \prod_{i=2}^j n_i & \text{ if }j>1.
\end{cases}\]
We shall set $D:=D_k$. In particular, when $n=n_1=n_2=n_3=\dots =n_k$, we find that $N_j=n^j$ and $D_j=n^{j-1}$.
\begin{lemma}
    With respect to notation above, the following assertions hold.
    \begin{enumerate}
        \item The degree of $F_j(x)$ (as a polynomial in $x$ with coefficients in $A$) is equal to $N_j$.
        \item The maximum total degree of $F_j$ in the variables $\{T_{u,v}\}$ is $\leq D_j$.
    \end{enumerate}
    In particular, the result implies that $F(x)$ is a polynomial of degree equal to $N$, and the maximum total degree of $F$ in the variables $\{T_{u, v}\}$ is $\leq D$.
\end{lemma}

\begin{proof}
    We prove the result by induction on $j$. When $j=1$, we find that \[F_1=g_1=x^{n_1}+\sum_{v=1}^{n_1} T_{1,v} x^{n_1-v}\] and the assertions are clear. Assume that $j\geq 2$, and write $F_j=F_{j-1} \wr g_j$. We find that 
    \[F_j(X)=g_j\left(F_{j-1}(x)\right)=F_{j-1}(x)^{n_j}+\sum_{v=1}^{n_j} T_{j,v} F_{j-1}(x)^{n_j-v}.\] Therefore, we find that 
    \[\op{deg}_x(F_j)=n_j\op{deg}_x (F_{j-1})\text{ and }\op{deg}_{\{T_{u,v}\}}(F_j)\leq n_j\op{deg}_{\{T_{u,v}\}} (F_{j-1}).\] Thus, the result follows by induction on $j$. 
\end{proof}
Odoni's study of arboreal representations led to a criterion for a polynomial to give rise to an $S(\vec{n})$-extensions. First, we recall a general criterion.

\begin{theorem}[Odoni]\label{odoni corollary}
    Let $\mathbf{F}$ be a field of characteristic $0$, and let $f(x)\in \mathbf{F}(x)$ be monic and square free, with $G=\op{Gal}(f(x)/\mathbf{F})$. For $l\geq 2$ let $\mathfrak{G}(x)$ be the generic monic polynomial 
    \[\mathfrak{G}(x)=x^l+a_1 x^{l-1}+a_2x^{l-2}+\dots +a_{l-1} x+a_l.\] Then, $f(\mathfrak{G}(x))$ is squarefree in $\mathbf{F}(a_1, \dots, a_l)[x]$ and 
    \[\op{Gal}(f(\mathfrak{G}(x))/\mathbf{F}(a_1, \dots, a_l))=G\wr S_n.\]
\end{theorem}
\begin{proof}
The result is above is \cite[Corollary 8.4]{odoni1985galois}.
\end{proof}

\begin{corollary}\label{odoni contribution}
Let $\vec{n}=(n_1, \dots, n_k)\in \Z_{\geq 2}^k$, and recall from \eqref{defn of A} that \[A:=\cO_K[\{T_{u,v}\mid u\in [1, k], v\in [1, n_i]\}].\] Let $\mathbf{L}$ be the fraction field of $A$ and 
\[F:=g_1\wr g_2\wr \dots \wr g_k=f_{1,n_1}\wr f_{2, n_2}\wr \dots \wr f_{k, n_k}\in A[x],\] as in \eqref{def of F}. Then, the following assertions hold
\begin{enumerate}
    \item $F(x)$ is an irreducible polynomial of degree $N=n_1n_2\dots n_k$ over $\mathbf{L}$,
    \item $\op{Gal}(F(x)/\mathbf{L})$ is isomorphic to $S(\vec{n})=S_{n_1}\wr S_{n_2}\wr \dots \wr S_{n_k}$.
\end{enumerate}
\end{corollary}
\begin{proof}
    The result follows by induction on $k$, and is an easy consequence of Theorem \ref{odoni corollary}.
\end{proof}

We identify the $\cO_K$-valued points of $\mathfrak{X}=\op{Spec}A$ with $\mathfrak{X}(\cO_K)=\cO_K^m$, where $m:=\sum_{i=1}^k n_i$. For $\alpha=(\alpha_{u,v})\in \mathfrak{X}(\cO_K)$, let $\Phi_\alpha: A[x]\rightarrow \cO_K[x]$ be the map induced by specializing $T_{u,v}\mapsto \alpha_{u,v}$. We shall set \[F_{j,\alpha}(x):=\Phi_\alpha(F_j(x)), g_{j, \alpha}(x):=\Phi_\alpha(g_j(x))\text{ and }F_{\alpha}(x):=\Phi_\alpha(F(x)).\] Note that $G_\alpha:=\op{Gal}(F_\alpha(x)/K)$ is a subgroup of $G$. By the integral version of the Hilbert irreducibility theorem, for most points $\alpha\in \mathfrak{X}(\cO_K)$, we have that $G_\alpha=G$. 

\begin{definition}
    With respect to notation above, we set
    \[\mathfrak{X}(\cO_K;Y):=\left\{ \alpha\in \mathfrak{X}(\cO_K)\mid \lVert \alpha_{u,v}\rVert \leq Y^{v \left(\prod_{i=1}^{u-1}n_i\right)}=Y^{v N_{u-1}}\text{ for all coordinates } \alpha_{i,j}
 \right\},\] where it is understood that $\prod_{i=1}^{0}n_i:=1$.
\end{definition}

\begin{proposition}\label{lemma 5.4}
    Let $Y>0$ and $\alpha\in \mathfrak{X}(\cO_K;Y)$. Then, for some suitably large constant $C_{1, j}>0$, which depends only on $j$, $\vec{n}$ and $K$, we have that 
    \[\lVert F_{j,\alpha}\rVert \leq C_{1, j} Y.\] Setting $C_1:=\op{max}\{C_{1, j}\mid j\in [1, k]\}$, we find that for all $j\in [1, k]$, \[\lVert F_{j,\alpha}\rVert \leq C_1Y.\] In particular, $\lVert F_{\alpha}\rVert \leq C_1Y$.
\end{proposition}
\begin{proof}
    We prove the result by induction on $j$. The case when $j=1$ is clear and thus assume that $j\geq 2$. For ease of notation, $f_j:=g_{j, \alpha}$, and thus write $F_{j,\alpha}=f_1\wr f_2\wr \dots \wr f_j$. We express $F_{j,\alpha}$ as $H_1\wr H_2$, where $H_1=F_{j-1,\alpha}=f_1\wr f_2\wr \dots \wr f_{j-1}$ and $H_2=f_j$. By inductive hypothesis, $\lVert H_1\rVert \leq C_{1, j-1}Y$. We write 
    \[\begin{split}
    & H_1=x^{m}+b_1 x^{m-1}+b_2 x^{m-2}+\dots +b_{m-1} x+b_m \\
    & H_2=x^{n}+a_1 x^{n-1}+a_2 x^{n-2}+\dots +a_{n-1} x+a_n, \\
    \end{split}\]
    where $m=N_{j-1}$ and $n=n_j$. Since $\lVert H_1\rVert \leq C_{1, j-1}Y$, we note that $\lVert b_i\rVert \leq (C_{1, j-1}Y)^i$ for all $i$. On the other hand, $a_i=\alpha_{j, i}$ and therefore,
    \[\lVert a_i \rVert\leq Y^{iN_{j-1}}=Y^{im}.\] 
    \par We find that 
    \[F_{j, \alpha}=\sum_{i=0}^n a_i \left(b_0 x^{m}+b_1 x^{m-1}+b_2 x^{m-2}+\dots +b_{m-1} x+b_m\right)^{n-i},
    \]
    where $a_0:=1$ and $b_0:=1$.
    Let's consider the expression \[a_i \left(b_0x^{m}+b_1 x^{m-1}+b_2 x^{m-2}+\dots +b_{m-1} x+b_m\right)^{n-i},\] i.e., the $i$-th term in the sum above. This polynomial is of degree $m(n-i)$ in $x$ and is a sum of monomials of the form
    \[ a_i b_{m-t_1}x^{t_1}b_{m-t_2}x^{t_2}\dots b_{m-t_{n-i}}x^{t_{n-i}}=a_i (b_{m-t_1}\dots b_{m-t_{n-i}})x^{\sum_{k=1}^{n-i} t_k}. \]
    We find that 
    \[\lVert a_i (b_{m-t_1}\dots b_{m-t_{n-i}})\rVert \leq Y^{im}(C_{1,j-1}Y)^{m(n-i)-\sum_{k=1}^{n-i} t_k}\ll Y^{mn-\sum_{k=1}^{n-i} t_k}= Y^{N_j-\sum_{k=1}^{n-i} t_k}.\] Therefore, there is a large enough constant $C_{1, j}>0$, such that $\lVert F_{j,\alpha}\rVert \leq C_{1, j} Y$. This completes the inductive step.
\end{proof}
It is convenient to explicitly state an immediate consequence of the above result.
\begin{corollary}\label{cor 5.7}
Let $Y>0$, $\alpha\in \mathfrak{X}(\cO_K;Y)$ and let $C_1$ be the constant defined according to Proposition \ref{lemma 5.4}. Then, for all $j\in [1, k]$, we have that
\[\lVert F_{j,\alpha}(0)\rVert \leq C_2 Y^{N_j},\] where, $C_2:=C_1^{N}$.
\end{corollary}
\begin{proof}
    The above is a direct consequence of Proposition \ref{lemma 5.4}.
\end{proof}

\begin{definition}\label{def of X'}
    Let $\mathfrak{X}'(\cO_K;Y)$ be the set of $\alpha\in \mathfrak{X}(\cO_K;Y)$, such that $F_\alpha(x)$ is irreducible and \[G_\alpha:=\op{Gal}(F_\alpha(x)/K)\simeq S(\vec{n}).\]
\end{definition}
We note that by construction, $G_\alpha$ is identified with a subgroup of $S(\vec{n})$. From the Hilbert irreducibility theorem, we obtain an asymptotic estimate for $\#\mathfrak{X}'(\cO_K;Y)$.
\begin{proposition}\label{prop A}
    Let $\mathfrak{X}'(\cO_K;Y)$ be as in Definition \ref{def of X'}. Then,
    \[\# \mathfrak{X}'(\cO_K;Y)\gg Y^{d A},\]
    where 
    \[A=A(\vec{n}):=\sum_{j=1}^{k} \left(\frac{n_j+1}{2}\right)\left(\prod_{v=1}^j n_v\right).\]
\end{proposition}
    \begin{proof}
        We find that 
        \[\# \mathfrak{X}(\cO_K;Y)\gg \prod_{i=1}^k\prod_{j=1}^{n_i}Y^{djN_{i-1}}=Y^{dA},\]
        where we recall that $$N_i:=\begin{cases}
            \prod_{v=1}^{i} n_v & \text{ if }i>0;\\
            1 & \text{ if }i=0.
        \end{cases}$$
        It then follows from Theorem \ref{Hilbert irred} and Corollary \ref{odoni contribution} that
        \[\# \mathfrak{X}'(\cO_K;Y)\gg Y^{d A}.\]
    \end{proof}

We use the asymptotic lower bound for $\#\mathfrak{X}'(\cO_K;Y)$ to obtain a lower bound for $\# \cP_{N, K}\left(Y; S(\vec{n})\right)$, where, we recall that 
\[\cP_{N, K}\left(Y; S(\vec{n})\right):=\{z\in \cO_{\bar{K}}\mid \lVert z\rVert \leq Y, [K(z):K]=N, \op{Gal}(\widetilde{K(z)}/K)\simeq S(\vec{n})\}.\]

Next, we define a map 
\[\Psi: \mathfrak{X}'(\cO_K;Y)\rightarrow \cP_{N, K}\left(C_3 Y; S(\vec{n})\right),\]where $C_3>0$ is a suitably large constants which we shall specify. It is convenient to first state a basic result that relates the height $\lVert z\rVert$ of $z\in \cO_{\bar{K}}$ to the height of its minimal polynomial.
\begin{lemma}\label{lemma 5.9}
    Given $z\in \cO_{\bar{K}}$, let $f(x)$ be the minimal polynomial of $z$ over $\cO_K$, then, $ \lVert z\rVert \leq 2 \lVert f\rVert$.
\end{lemma}
\begin{proof}
     By way of contradiction, $\lVert z\rVert > 2 \lVert f\rVert$, we set $u:=\frac{\lVert f\rVert}{\lVert z\rVert}$ and note that $u\in (0,1/2)$. Write $f(x)=x^n+a_1x^{n-1}+\dots+a_n$. We note that $\lVert a_i\rVert \leq \lVert f\rVert^i$. Thus, we find that 
    \[\lVert z\rVert^n\leq \sum_{i=1}^n \lVert f\rVert^i \lVert z\rVert^{n-i}=\lVert z\rVert^{n}\left(u+u^2+\dots+u^n\right)<\lVert z\rVert^{n}\left(\frac{u}{1-u}\right). \] Therefore, $\left(\frac{u}{1-u}\right)>1$, i.e., $u>1/2$, a contradiction.
\end{proof}

\par Recall that from Definition \ref{def of X'} that for $\alpha\in  \mathfrak{X}'(\cO_K;Y)$, we have that $F_\alpha:=\Phi_\alpha(F)$ is irreducible of degree $N$, and $G_\alpha\simeq S(\vec{n})$. Choose a root $z_\alpha$ of $F_\alpha(x)$. Note that $\widetilde{K(z_\alpha)}$ is an extension of $K$ for which $G_\alpha=\op{Gal}(\widetilde{K(z_\alpha)}/K)$. Note that the choice of root $z_\alpha$ is non-canonical, however, we make one such choice for each $\alpha$. Then, we set $\Psi(\alpha):=z_\alpha$. It follows from Proposition \ref{lemma 5.4} that $\lVert F_\alpha\rVert \leq C_1 Y$. Setting $C_3:=2C_1$, it follows from Lemma \ref{lemma 5.9} that $\lVert z_\alpha\rVert \leq C_3 Y$. 
Thus, we obtain a map 
\[\Psi: \mathfrak{X}'(\cO_K;Y)\rightarrow \cP_{N, K}\left(C_3 Y; S(\vec{n})\right).\]
Note that the map defined above is non-canonical, since it depends on a choice of root $z_\alpha\in \bar{K}$ of $F_\alpha(x)\in K[x]$ for each $\alpha\in \mathfrak{X}'(\cO_K;Y)$. 

\par In order to better describe the cardinality of the fibers of the above map $\Psi$, we write it as a composite of two maps which we shall now describe.
\begin{definition}
    Let $\cP_{N, K}'(Y;S(\vec{n}))$ be the set of tuples $(z; E_0, \dots, E_k)$ such that $z\in \cP_{N, K}(Y;S(\vec{n}))$ and $E_i$ are a tower of fields
\[K=E_k\subset E_{k-1}\subset E_{k-2}\subset \dots \subset E_{1}\subset E_0=K(z),\] such that for all $j\in [0, k-1]$, there is an isomorphism \[\op{Gal}(\widetilde{E_j}/K)\simeq S(n_1 ,\dots, n_{k-j}).\]
\end{definition}
For $\alpha\in \mathfrak{X}'(\cO_K;Y)$, set  $g_{i, \alpha}:=\Phi_\alpha(g_i)$, where we recall that $\Phi_\alpha$ is the evaluation map $ A[x]\rightarrow \cO_K[x]$ induced by specializing $T_{u,v}\mapsto \alpha_{u,v}$, and that 
\[g_i(x)=x^{n_i}+\sum_{v=1}^{n_i} T_{i,v} x^{n_i-v}.\]
Thus, we have that $g_{i,\alpha}(x)=x^{n_i}+\sum_{v=1}^{n} \alpha_{i,v} x^{n_i-v}$. We define a sequence of elements $z_j(\alpha)$ inductively as follows, $z_0(\alpha)=z_\alpha$ and $z_{j+1}(\alpha)=g_{j+1, \alpha}\left(z_j(\alpha)\right)$, as illustrated below
\[z_\alpha=z_0(\alpha)\xrightarrow{g_{1, \alpha}}z_1(\alpha)\xrightarrow{g_{2, \alpha}}\dots \xrightarrow{g_{k-1, \alpha}}z_{k-1}(\alpha)\xrightarrow{g_{k, \alpha}}z_k(\alpha)=0.\]Also, we set $K_{j, \alpha}$ to denote the field $K(z_{j}(\alpha))$. Thus, we have a tower of fields
\[K=K_{k, \alpha}\subset K_{k-1, \alpha}\subset \dots \subset K_{0, \alpha}=K(z_\alpha).\]
Note that for $j\in [0,k-1]$, the isomorphism $G_\alpha\simeq S(\vec{n})$ induces an isomorphism \[\op{Gal}(\widetilde{K_{j, \alpha}}/K)\simeq S(n_1, n_2, \dots, n_{k-j}).\] Let $\mathcal{B}(Y)$ be the set of tuples $(u_1, \dots, u_{k-1})\in \cO_K^{(k-1)}$ such that $\lVert u_i\rVert \leq Y^{ N_i}$ for all coordinates $u_i$. Therefore, we find that \begin{equation}\label{BY eqn}\#\mathcal{B}(Y)\sim Y^{d\left(\sum_{i=1}^{k-1} N_i\right)}.\end{equation}
\begin{definition}
    With respect to notation above, we define the map
\[\Psi':\mathfrak{X}'(\cO_K;Y)\longrightarrow \cP_{N, K}'(C_3 Y;S(\vec{n})) \times \mathcal{B}(C_2 Y)\]as follows
\[\begin{split}
    & \Psi'(\alpha):=(\Psi_1'(\alpha),\Psi_2'(\alpha)), \text{ where,}\\
&\Psi_1'(\alpha):=\left(z_\alpha, K_{0, \alpha}, K_{1,\alpha}, \dots, K_{k, \alpha}\right);\\
&\Psi_2'(\alpha):= \left( F_{1, \alpha}(0), F_{2, \alpha}(0), \dots, F_{k-1, \alpha}(0)\right).
\end{split}\]
\end{definition}
\begin{proposition}
    The map $\Psi'$ above is well defined.
\end{proposition}
\begin{proof}
It suffices to show that the maps $\Psi_1'$ and $\Psi_2'$ are well defined. For $\alpha\in \mathfrak{X}'(\cO_K;Y)$, the polynomial $F_\alpha(x)$ is irreducible and $\op{Gal}(F_\alpha(x)/K)$. It follows from Lemma \ref{lemma 5.9} that 
\[\lVert z_\alpha\rVert \leq 2\lVert F_\alpha\rVert .\] On the other hand, it follows from Proposition \ref{lemma 5.4} that 
\[\lVert F_\alpha\rVert\leq C_1 Y,\] and therefore, 
\[\lVert z_\alpha\rVert \leq C_2 Y,\] where we recall that $C_2:=2 C_1$. Since $\op{Gal}(F_\alpha(x)/K)\simeq S(\vec{n})$, it follows that
\[\op{Gal}(\widetilde{K_{j, \alpha}}/K)\simeq S(n_1, \dots, n_{k-j}).\] Therefore, $\Psi_1'(\alpha)$ is an element in $\cP_{N, K}'(C_3 Y;S(\vec{n}))$.
\par Corollary \ref{cor 5.7} asserts that for all $j\in [1, k]$, we have that
\[\lVert F_{j,\alpha}(0)\rVert \leq C_2 Y^{N_j}.\] Therefore, $\Psi_2'(\alpha)\in \mathcal{B}(C_2 Y)$ and thus the map $\Psi$ is shown to be well defined.
\end{proof}
\begin{lemma}\label{dumb lemma}
    For $\alpha\in \mathfrak{X}'(\cO_K;Y)$ and $j\in [0, k-1]$, the minimal polynomial of $z_\alpha$ over $K_{j, \alpha}$ is $G_{j, \alpha}(x):=F_{j, \alpha}(x)-z_j(\alpha)$.\end{lemma}
\begin{proof}
    Since $F_{j, \alpha}(z_\alpha)=z_j(\alpha)$, it is clear that $z_\alpha$ satisfies $G_{j, \alpha}(x)$. Let $\gamma_1, \dots, \gamma_t$ be the roots of $G_{j, \alpha}(x)$ (with repetitions). Since $\op{Gal}(F_\alpha(x)/K)\simeq S(\vec{n})$, it follows that the roots $\gamma_i$ are all distinct, and $\op{Gal}(\bar{K}/K_{j, \alpha})$ acts transitively on $\{\gamma_1, \dots, \gamma_t\}$. Therefore, $G_{j, \alpha}(x)$ is irreducible over $K_{j, \alpha}$, and hence, it is the minimal polynomial of $z_\alpha$ over $K_{j, \alpha}$.
\end{proof}

\begin{lemma}\label{boring lemma}
    Let $\alpha$ and $\alpha'$ be elements in $\mathfrak{X}'(\cO_K; Y)$. Assume that for all $j\in [1, k]$, there is an equality of polynomials $F_{j, \alpha}(x)=F_{j, \alpha'}(x)$. Then, we have that $\alpha=\alpha'$.
\end{lemma}
\begin{proof}
    We recall that $F_{j, \alpha}=g_{1, \alpha}\wr g_{2, \alpha}\wr \dots \wr g_{k, \alpha}$, where 
    \[g_{i, \alpha}=x^{n_i}+\sum_{v=1}^{n_i} \alpha_{i,v} x^{n_i-v}.\] It suffices to show that for all $j\in [1, k]$, 
    \[g_{j, \alpha}(x)=g_{j, \alpha'}(x).\] We prove this equality by induction on $j$. For $j=1$, we have that \[g_{1, \alpha}(x)=F_{1, \alpha}(x)=F_{1, \alpha'}(x)=g_{1, \alpha'}(x).\] Therefore, we assume that $j\geq 2$. Let $H(x):=F_{j-1, \alpha}(x)=F_{j-1, \alpha'}(x)$, then, we find that 
    \[g_{j, \alpha}\left(H(x)\right)=g_{j, \alpha}\left(F_{j-1, \alpha}(x)\right)=F_{j, \alpha}(x)= F_{j, \alpha'}(x)=g_{j, \alpha'}\left(F_{j-1, \alpha'}(x)\right)=g_{j, \alpha'}\left(H(x)\right).\]
    For ease of notation, set $a_v:=\alpha_{j,v}$ and $a_v':=\alpha_{j,v}'$. Write $H(x)=\sum_{i=0}^m b_i x^{m-i}$, and thus, 
    \begin{equation}\label{eq 1}0=g_{j, \alpha}\left(H(x)\right)-g_{j, \alpha'}\left(H(x)\right)=\sum_{v=1}^{n_j}(a_v-a_v')H(x)^{n_j-v}.\end{equation} Assume by way of contradiction that for some $v\in [1, n_j]$, we have that $a_v\neq a_v'$. Let $v_0$ be the minimum value in the range $[1, n_j]$ for which $a_{v_0}\neq a_{v_0}'$. Then, the degree of $\sum_{v=1}^{n_j}(a_v-a_v')H(x)^{n_j-v}$ is equal to $(n_j-v_0)\op{deg} H(x)=(n_j-v_0)N_{j-1}$. Therefore, this implies that $v_0=n_j$. However, \eqref{eq 1} then implies that $(a_{n_j}-a_{n_j}')=0$, a contradiction. This completes the inductive step and the proof of the result.
\end{proof}

\begin{proposition}\label{prop 5.9}
    The map $\Psi'$ above is an injection.
\end{proposition}
\begin{proof}
    Let $\alpha$ and $\alpha'$ be such that \[\Psi'(\alpha)=\Psi'(\alpha').\]Since $\Psi_1'(\alpha)=\Psi'_1(\alpha')$, it follows that $z_\alpha=z_{\alpha'}$ and $K_{j, \alpha}=K_{j, \alpha'}$ for all values of $j$. Lemma \ref{dumb lemma} asserts that the irreducible polynomial of $z_\alpha$ over $K_{j, \alpha}$ is \[F_{j, \alpha}(x)-F_{j, \alpha}(z_\alpha)=F_{j, \alpha}(x)-z_j(\alpha).\] Therefore, $\Psi_1'(\alpha)$ determines all non-constant coefficients of $F_{j, \alpha}(x)$. The equality $\Psi_1'(\alpha)=\Psi_1'(\alpha')$ implies that $F_{j, \alpha}(x)-F_{j, \alpha'}(x)$ is a constant. On the other hand, the equality $\Psi_2'(\alpha)=\Psi_2'(\alpha')$ implies that for all $j\in [1, k-1]$ \[F_{j, \alpha}(0)=F_{j, \alpha'}(0).\] Therefore, we have shown that for all $j\in [1, k-1]$, \[F_{j, \alpha}(x)=F_{j, \alpha'}(x).\] Since $z_\alpha=z_{\alpha'}$, it follows that the minimal polynomials $F_\alpha(x)$ and $F_{\alpha'}(x)$ are equal. Hence, for all $j\in [1, k]$, \[F_{j, \alpha}(x)=F_{j, \alpha'}(x).\] Lemma \ref{boring lemma} then implies that $\alpha=\alpha'$, and therefore, $\Psi'$ is injective.
\end{proof}
\section{Proof of the main result}\label{s 6}
\par In this short section, we prove the main result of the paper. First, we establish an asymptotic lower bound for $\cP_{N, K}\left(Y; S(\vec{n})\right)$.
\begin{proposition}\label{prop 5.5}
    Let $k\in \Z_{\geq 2}$ and $\vec{n}=(n_1, \dots, n_k)\in \Z_{\geq 2}^k$, we have that 
    \[\cP_{N, K}\left(Y; S(\vec{n})\right)\gg Y^{dB(\vec{n})},\] where 
    \[B(\vec{n}):=A(\vec{n})-\sum_{j=1}^{k-1} N_j=\sum_{j=1}^{k-1} \left(\frac{n_j-1}{2}\right)\left(\prod_{v=1}^j n_v\right)+\left(\frac{n_k+1}{2}\right)\left(\prod_{v=1}^k n_v\right).\]
\end{proposition}
\begin{proof}
By Proposition \ref{prop A},
\[\# \mathfrak{X}'(\cO_K;Y)\gg Y^{d A},\]
    where 
    \[A=A(\vec{n}):=\sum_{j=1}^{k} \left(\frac{n_j+1}{2}\right)\left(\prod_{v=1}^j n_v\right).\]
    Recall that by \eqref{BY eqn}, $\#\mathcal{B}(Y)\sim Y^{d\left(\sum_{i=1}^{k-1} N_i\right)}$.
    It follows from Proposition \ref{prop 5.9} that $\Psi'$ is injective, and hence, 
\[\#\cP_{N, K}'(Y;S(\vec{n}))\gg \frac{\#\mathfrak{X}'(\cO_K; Y)}{\#\mathcal{B}( Y)}\gg Y^{d \left(A(\vec{n})-\sum_{i=1}^{k-1}N_i\right)}=Y^{d B(\vec{n})}.\]
There is a constant $C$ which depends only on $\vec{n}$ such that \[C\#\cP_{N, K}(Y;S(\vec{n}))\geq \# \cP_{N, K}'(Y;S(\vec{n})).\] By the Galois correspondence, $C$ can be taken to be the number of towers $\{1\}=H_0\subset H_1\subset H_2\subset \dots H_{k-1}\subset H_k=S(\vec{n})$ of subgroups $H_i$ of $S(\vec{n})$. Clearly, $C$ depends only on $\vec{n}$. Therefore, we have shown that 
\[\cP_{N, K}\left(Y; S(\vec{n})\right)\gg Y^{dB(\vec{n})},\] where 
    \[B(\vec{n}):=\sum_{j=1}^{k-1} \left(\frac{n_j-1}{2}\right)\left(\prod_{v=1}^j n_v\right)+\left(\frac{n_k+1}{2}\right)\left(\prod_{v=1}^k n_v\right).\]
\end{proof}

\begin{proof}[Proof of Theorem \ref{main thm}]
    We first verify that the hypotheses of Proposition \ref{prop 4.3} are
    satisfied. Recall that
    \[
    B
    =
    \sum_{j=1}^{k-1}
    \left(\frac{n_j-1}{2}\right)
    \left(\prod_{v=1}^j n_v\right)
    +
    \left(\frac{n_k+1}{2}\right)
    \left(\prod_{v=1}^k n_v\right)
    \]
    and
    \[
    N=\prod_{v=1}^k n_v.
    \]
    Since $k\geq 2$ and each $n_j\geq 2$, the first sum in the definition of
    $B$ is positive. Hence
    \[
    B>
    \left(\frac{n_k+1}{2}\right)
    \left(\prod_{v=1}^k n_v\right)
    =
    \left(\frac{n_k+1}{2}\right)N.
    \]
    Since $n_k\geq 2$, it follows that
    \[
    B>\frac{3}{2}N>N.
    \]
    Thus the condition $C>n$ in Proposition \ref{prop 4.3} is satisfied after
    replacing $n$ by $N$ and $C$ by $B$.

    By Proposition \ref{prop 5.5}, applied to the polynomial constructed using
    Odoni's theorem, we have
    \[
    \#\cP_{N,K}(Y;S(\vec n))\gg Y^{dB}.
    \]
    Therefore Proposition \ref{prop 4.3} may be applied with
    \[
    n=N,\qquad G=S(\vec n),\qquad C=B.
    \]

    We now use the Ellenberg--Venkatesh upper bound
    \[
    N_{N,K}(X)\ll_{N,K} X^{\op{exp}(A\sqrt{\log N})}
    \]
    for an absolute constant $A>0$. Since
    \[
    N_{N,K}(X;S(\vec n))\leq N_{N,K}(X),
    \]
    we may take
    \[
    e=\op{exp}(A\sqrt{\log N})
    \]
    in Proposition \ref{prop 4.3}. Notice that $e>1$, and hence in particular
    $e\geq 1/(N-1)$.

    Proposition \ref{prop 4.3} now gives two cases. If
    \[
    B\geq N\left(\op{exp}(A\sqrt{\log N})+\frac12\right),
    \]
    then
    \[
    N_{N,K}(X;S(\vec n))
    \gg
    X^{\frac{B-N/2}{N^2-N}}.
    \]
    If instead
    \[
    B\leq N\left(\op{exp}(A\sqrt{\log N})+\frac12\right),
    \]
    then
    \[
    N_{N,K}(X;S(\vec n))
    \gg
    X^{\frac{2\op{exp}(A\sqrt{\log N})(B-N)}
    {(2\op{exp}(A\sqrt{\log N})-1)(N^2-N)}}.
    \]
    This is exactly the asserted lower bound.
\end{proof}

\begin{proof}[Proof of Theorem \ref{main corollary}]
    We apply Theorem \ref{main thm} in the special case
    \[
    \vec n=(n,n,\dots,n).
    \]
    Then
    \[
    N=N(\vec n)=\prod_{v=1}^k n=n^k.
    \]
    We next compute \(B=B(\vec n)\). By definition,
    \[
    B
    =
    \sum_{j=1}^{k-1}
    \left(\frac{n-1}{2}\right)n^j
    +
    \left(\frac{n+1}{2}\right)n^k.
    \]
    Since
    \[
    \sum_{j=1}^{k-1}n^j
    =
    \frac{n^k-n}{n-1},
    \]
    we get
    \[
    \sum_{j=1}^{k-1}
    \left(\frac{n-1}{2}\right)n^j
    =
    \frac{n^k-n}{2}.
    \]
    Hence
    \[
    B
    =
    \frac{n^k-n}{2}
    +
    \frac{(n+1)n^k}{2}
    =
    \frac{n^{k+1}+2n^k-n}{2}.
    \]

    Set
    \[
    E_{n,k}:=\op{exp}\left(A\sqrt{k\log n}\right).
    \]
    Since \(N=n^k\), we have
    \[
    \op{exp}\left(A\sqrt{\log N}\right)
    =
    \op{exp}\left(A\sqrt{k\log n}\right)
    =
    E_{n,k}.
    \]
    The first case in Theorem \ref{main thm} occurs precisely when
    \[
    B\geq N\left(E_{n,k}+\frac12\right).
    \]
    Dividing by \(N=n^k\), this condition becomes
    \[
    \frac{B}{N}\geq E_{n,k}+\frac12.
    \]
    But
    \[
    \frac{B}{N}
    =
    \frac{n^{k+1}+2n^k-n}{2n^k}
    =
    \frac{n+2-n^{1-k}}{2}.
    \]
    Therefore the first case is equivalent to
    \[
    \frac{n+1-n^{1-k}}{2}\geq E_{n,k}.
    \]
    In this range, Theorem \ref{main thm} gives
    \[
    N_{n^k,K}(X;[S_n]^k)
    \gg
    X^{\frac{B-N/2}{N^2-N}}.
    \]
    Now
    \[
    B-\frac{N}{2}
    =
    \frac{n^{k+1}+2n^k-n}{2}
    -
    \frac{n^k}{2}
    =
    \frac{n^{k+1}+n^k-n}{2},
    \]
    and
    \[
    N^2-N=n^{2k}-n^k.
    \]
    Thus the exponent in the first range is
    \[
    \delta_{n,k}
    =
    \frac{n^{k+1}+n^k-n}{2(n^{2k}-n^k)}.
    \]

    On the other hand, if
    \[
    \frac{n+1-n^{1-k}}{2}\leq E_{n,k},
    \]
    then we are in the second case of Theorem \ref{main thm}. Hence
    \[
    N_{n^k,K}(X;[S_n]^k)
    \gg
    X^{\frac{2E_{n,k}(B-N)}
    {(2E_{n,k}-1)(N^2-N)}}.
    \]
    We compute
    \[
    B-N
    =
    \frac{n^{k+1}+2n^k-n}{2}
    -
    n^k
    =
    \frac{n^{k+1}-n}{2}.
    \]
    Therefore the exponent in the second range is
    \[
    \delta'_{n,k}
    =
    \frac{E_{n,k}(n^{k+1}-n)}
    {(2E_{n,k}-1)(n^{2k}-n^k)}.
    \]
    Since
    \[
    n^{k+1}-n=n(n^k-1)
    \]
    and
    \[
    n^{2k}-n^k=n^k(n^k-1),
    \]
    this may also be written as
    \[
    \delta'_{n,k}
    =
    \frac{E_{n,k}}{2E_{n,k}-1}\cdot \frac{1}{n^{k-1}}.
    \]
    This proves the asserted case decomposition.

    Finally, in the first range we have
    \[
    \delta_{n,k}
    =
    \frac{n^{k+1}+n^k-n}{2(n^{2k}-n^k)}
    \geq
    \frac{n^k}{2n^k(n^{k-1})}
    =
    \frac{1}{2n^{k-1}},
    \]
    and in the second range,
    \[
    \delta'_{n,k}
    =
    \frac{E_{n,k}}{2E_{n,k}-1}\cdot \frac{1}{n^{k-1}}
    >
    \frac{1}{2n^{k-1}}.
    \]
    Hence in both cases
    \[
    N_{n^k,K}(X;[S_n]^k)
    \gg
    X^{\frac{1}{2n^{k-1}}}.
    \]
\end{proof}

\begin{remark}\label{only remark}
    When $K=\Q$, one may also apply the Theorem \ref{ptbw} of Pierce, Turnage-Butterbaugh and Wood and obtain that 
    \[N_{n^k, \Q}(X;[S_n]^k)\gg_{\epsilon} X^{\beta_{n, k}-\epsilon},\] where $\beta_{n, k}:=\frac{1-(n!)^{-k}}{n^{k-1}(2n^k-2)}$. This lower bound is weaker than Theorem \ref{main corollary} above, and this is why we did not pursue this particular strategy. Our method is significantly more elaborate and gives better bounds, and applies to all number fields $K$.
\end{remark}

\bibliographystyle{alpha}
\bibliography{references}
\end{document}